\documentclass[12pt]{article} 
\usepackage{amsfonts,amsmath,latexsym,amssymb,mathrsfs,amsthm,comment}
\usepackage{slashbox}
\usepackage{caption}

\let\OLDthebibliography\thebibliography
\renewcommand\thebibliography[1]{
  \OLDthebibliography{#1}
  \setlength{\parskip}{3pt}
  \setlength{\itemsep}{0pt plus 0.3ex}
}

\def\numberlikeadb{\global\def\theequation{\thesection.\arabic{equation}}}
\numberlikeadb
\newtheorem{theorem}{Theorem}[section]

\usepackage{lscape}
\usepackage{caption}
\usepackage{multirow}
\begin{document}

\title{A note on the distribution of the product of zero mean correlated normal random variables}
\author{Robert E. Gaunt\footnote{School of Mathematics, The University of Manchester, Manchester M13 9PL, UK}}

\date{} 
\maketitle

\vspace{-10mm}

\begin{abstract}The problem of finding an explicit formula for the probability density function of two zero mean correlated normal random variables dates back to 1936.  Perhaps surprisingly, this problem was not resolved until 2016.  This is all the more surprising given that a very simple proof is available, which is the subject of this note; we identify the  product of two zero mean correlated normal random variables as a variance-gamma random variable, from which an explicit formula for probability density function is immediate.  


\end{abstract}

\noindent{{\bf{Keywords:}}} Product of correlated normal random variables; probability density function; variance-gamma distribution

\noindent{{{\bf{AMS 2010 Subject Classification:}}} Primary 60E05; 62E15

\section{Introduction}

Let $(X,Y)$ be a bivariate normal random vector with zero mean vector, variances $(\sigma_X^2,\sigma_Y^2)$ and correlation coefficient $\rho$.  The exact distribution of the product $Z=XY$ has been studied since 1936 \cite{craig}; see \cite{np16} for an overview of some of the contributions.  The  distribution of $Z$ has been used in numerous applications since 1936, with some recent examples being: product confidence limits for indirect effects \cite{mac2}; statistics of Lagrangian power in two-dimensional turbulence \cite{bandi}; statistical mediation analysis \cite{mac1}.  However, despite this interest, the problem of finding an exact formula for the probability density function (PDF) of $Z$ remained open for many years.

Recently in 2016, some 80 years after the problem was first studied, an approach based on characteristic functions was used by \cite{np16} to obtain an explicit formula for the PDF of $Z$.  As a by-product, the exact distribution was obtained for the mean $\overline{Z}=\frac{1}{n}(Z_1+Z_2+\cdots+Z_n)$, where $Z_1, Z_2,\ldots, Z_n$ are independent and identical copies of $Z$.  This distribution is of interest itself; see, for example \cite{ware} for an application from electrical engineering.  Since the work of \cite{np16}, an exact formula for the PDF of a product of  correlated normal random variables with non-zero means was obtained by \cite{cui}.  This formula takes a complicated form, involving a double sum of modified Bessel functions of the second kind.

Shortly before the work of \cite{np16}, it was shown that $Z$ had a variance-gamma distribution (see Section 2 for further details regarding this distribution) in the thesis \cite{gaunt thesis} and the paper \cite{gaunt vg} (see part (iii) of Proposition 1.2).  A formula for the PDF is then immediate, and whilst not noted in those works, a formula for the PDF of $\overline{Z}$ can then be obtained from standard properties of the variance-gamma distribution.  In this note, we fill in this gap to provide a simple alternative proof of the main results of \cite{np16}.  Given the simplicity of our approach, it is surprising that such a natural problem in probability and statistics remained open for so many years.  Moreover, an advantage of our approach is that the distributions of $Z$ and $\overline{Z}$ are identified as being from the variance-gamma class, for which a well established distributional theory exists; see Chapter 4 of the book \cite{kkp01}.

The rest of this note is organised as follows.  In Section 2, we introduce the variance-gamma distribution and record some basic properties that will be useful in the sequel.  In Section 3, we provide an alternative proof of the main results of \cite{np16} by noting that $Z$ and $\overline{Z}$ are variance-gamma distributed.    

\section{The variance-gamma distribution}

The variance-gamma distribution with parameters $r > 0$, $\theta \in \mathbb{R}$, $\sigma >0$, $\mu \in \mathbb{R}$ has PDF
\begin{equation}\label{vgdef}f(x) = \frac{1}{\sigma\sqrt{\pi} \Gamma(\frac{r}{2})} \mathrm{e}^{\frac{\theta}{\sigma^2} (x-\mu)} \bigg(\frac{|x-\mu|}{2\sqrt{\theta^2 +  \sigma^2}}\bigg)^{\frac{r-1}{2}} K_{\frac{r-1}{2}}\bigg(\frac{\sqrt{\theta^2 + \sigma^2}}{\sigma^2} |x-\mu| \bigg), \quad x\in\mathbb{R}, 
\end{equation}
where the modified Bessel function of the second kind is given, for $x>0$, by $K_\nu(x)=\int_0^\infty \mathrm{e}^{-x\cosh(t)}\cosh(\nu t)\,\mathrm{d}t$.  If a random variable $W$ has density (\ref{vgdef}) then we write $W\sim \mathrm{VG}(r,\theta,\sigma,\mu)$.  This parametrisation was given in \cite{gaunt vg}.  It is similar to the parametrisation given by \cite{finlay} and alternative parametrisations are given by \cite{eberlein}, and the book \cite{kkp01} in which the name generalized Laplace distribution is used.  The distribution has semi-heavy tails, which are useful for modelling financial data \cite{madan}, and an overview of this and other applications are given in \cite{kkp01}.

We now review some basic properties of the variance-gamma distribution that will be needed in Section 3.  We stress that, with a standard handbook on definite integrals at hand (such as \cite{gradshetyn}), all that is required to establish these properties is a working knowledge of a first course in undergraduate probability.  Throughout, we shall set the location parameter $\mu$ equal to 0.  Firstly, we note a fundamental representation in terms of independent normal and gamma random variables (\cite{kkp01}, Proposition 4.1.2).  Let $S\sim\Gamma(\frac{r}{2},\frac{1}{2})$ (with PDF $\frac{1}{2^{r/2}\Gamma(r/2)}x^{r/2-1}\mathrm{e}^{-x/2}$, $x>0$) and $T\sim N(0,1)$ be independent.  Then 
\begin{equation}\label{lkl}\theta S+\sigma \sqrt{S}T\sim\mathrm{VG}(r,\theta,\sigma,0).
\end{equation} 
We will need the following special case of (\ref{lkl}). Let $U$ and $V$ be independent $N(0,1)$ random variables.  Then
\begin{equation}\label{vuelta}\theta U^2+\sigma UV\sim \mathrm{VG}(1,\theta,\sigma,0).
\end{equation}
This follows from (\ref{lkl}) due to the standard facts that $U^2\sim\Gamma(\frac{1}{2},\frac{1}{2})$ and $|U|V\stackrel{\mathcal{D}}{=}UV$.  Finally, we note that the class of variance-gamma distributions is closed under convolution (provided the random variables have common values of $\theta$ and $\sigma$) \cite{bibby}  and scaling by a constant. Let $W_1\sim \mathrm{VG}(r_1,\theta,\sigma,0)$ and $W_2\sim \mathrm{VG}(r_2,\theta,\sigma,0)$ be independent.  Then 
\begin{equation}\label{p1}W_1+W_2\sim \mathrm{VG}(r_1+r_2,\theta,\sigma,0). 
\end{equation} 
It is also clear from (\ref{lkl}) that 
\begin{equation}\label{p2}aW_1\sim \mathrm{VG}(r_1,a\theta,a\sigma,0). 
\end{equation}     

\section{Main result and proof}

Here, we provide an alternative proof of the main results of \cite{np16} by noting that $Z$ and $\overline{Z}$ are variance-gamma distributed.

\begin{theorem}Let $(X,Y)$ denote a bivariate normal random vector with zero means, variances $(\sigma_X^2,\sigma_Y^2)$ and correlation coefficient $\rho$.  

(i) Let $Z=XY$. Then $Z\sim \mathrm{VG}(1,\rho\sigma_X\sigma_Y, \sigma_X\sigma_Y\sqrt{1-\rho^2},0)$.

(ii) Let $Z_1, Z_2,\ldots, Z_n$ be independent random variables with the same distribution as $Z$. Let $\overline{Z}$ denote their sample mean.  Then $\overline{Z}\sim\mathrm{VG}(n,\frac{1}{n}\rho\sigma_X\sigma_Y, \frac{1}{n}\sigma_X\sigma_Y\sqrt{1-\rho^2},0)$.

(iii) Consequently, the PDFs of $Z$ and $\overline{Z}$ are given by
\begin{equation*}f_Z(x)=\frac{1}{\pi\sigma_X\sigma_Y\sqrt{1-\rho^2}}\exp\bigg(\frac{\rho x}{\sigma_X\sigma_Y(1-\rho^2)}\bigg)K_0\bigg(\frac{|x|}{\sigma_X\sigma_Y(1-\rho^2)}\bigg), \quad x\in\mathbb{R},
\end{equation*}
and, for $n\geq2$,
\begin{equation*}f_{\overline{Z}}(x)=\frac{n^{(n+1)/2}2^{(1-n)/2}|x|^{(n-1)/2}}{(\sigma_X\sigma_Y)^{(n+1)/2}\sqrt{\pi(1-\rho^2)}\Gamma\big(\frac{n}{2}\big)}\exp\bigg(\frac{\rho n x}{\sigma_X\sigma_Y(1-\rho^2)} \bigg)K_{\frac{n-1}{2}}\bigg(\frac{n |x|}{\sigma_X\sigma_Y(1-\rho^2)}\bigg), 
\end{equation*}
$x\in\mathbb{R}$, where $K_\nu(\cdot)$ is the modified Bessel function of the second kind of order $\nu$.

\end{theorem}

\begin{proof}We consider the case $\sigma_X=\sigma_Y=1$; the general case follows from (\ref{p2}).

(i) Define the random  variable $W$ by $W=\frac{1}{\sqrt{1-\rho^2}}(Y-\rho X)$.
It is straightforward to show that $W$ and $X$ are jointly standard normally distributed with correlation $0$.  Thus, $Z$ can be expressed in terms of independent $N(0,1)$ random variables $X$ and $W$ as follows:
\[Z=XY=X(\sqrt{1-\rho^2} W+\rho X)=\sqrt{1-\rho^2}XW+\rho X^2.\]
Hence, from (\ref{vuelta}), we have that $Z\sim \mathrm{VG}(1,\rho, \sqrt{1-\rho^2},0)$.

(ii) Use part (i) and properties (\ref{p1}) and (\ref{p2}) of variance-gamma random variables.

(iii) Combine (\ref{vgdef}) with parts (i) and (ii) and use the formula $\Gamma(\frac{1}{2})=\sqrt{\pi}$. 
\end{proof}

\subsection*{Acknowledgements}
The author is supported by a Dame Kathleen Ollerenshaw Research Fellowship.

\end{document}